\documentclass[11pt]{amsart}

\usepackage{graphicx,amssymb,latexsym,amsfonts,txfonts,amsmath,amsthm}
\usepackage{pdfsync,color}

\advance\hoffset by -0.9 truecm
\newcommand{\s}{\vspace{0.3cm}}

\newtheorem{theo}{Theorem}
\newtheorem{prop}{Proposition}
\newtheorem{coro}{Corollary}
\newtheorem{lemm}{Lemma}

\theoremstyle{remark}
\newtheorem{rema}{\bf Remark}

\begin{document}

\title{Uniformizations of stable $(\gamma,n)$-gonal Riemann surfaces}

\author{Ruben A. Hidalgo}
\address{Departamento de Matem\'atica y Estad\'{\i}stica, Universidad de La Frontera. Casilla 54-D, 4780000 Temuco, Chile}
\email{ruben.hidalgo@ufrontera.cl}

\thanks{Partially supported by Project Fondecyt 1150003 and Project Anillo ACT1415 PIA CONICYT}

\maketitle

\begin{center}
\dedicatory{To the memory of Alexander Vasil'ev}
\end{center}


\begin{abstract}
A $(\gamma,n)$-gonal pair is a pair $(S,f)$, where $S$ is a closed Riemann surface and $f:S \to R$ is a degree $n$ holomorphic map onto a closed Riemann surface $R$ of genus $\gamma$. If the signature of $(S,f)$ is of hyperbolic type, then there is pair $(\Gamma,G)$, called an uniformization of $(S,f)$, where $G$ is a Fuchsian group acting on the unit disc ${\mathbb D}$ containing $\Gamma$ as an index $n$ subgroup, so that $f$ is induced by the inclusion of $\Gamma <G$. The uniformization is uniquely determined by $(S,f)$, up to conjugation by holomorphic automorphisms of ${\mathbb D}$, and it permits 
to provide natural complex orbifold structures on the Hurwitz spaces parametrizing (twisted) isomorphic classes of pairs topologically equivalent to $(S,f)$. In order to produce certain compactifications of these Hurwitz spaces, one needs to consider the so called 
stable $(\gamma,n)$-gonal pairs, which are natural geometrical deformations of $(\gamma,n)$-gonal pairs. Due to the above, it seems interesting to search for 
uniformizations of stable $(\gamma,n)$-gonal pairs, in terms of certain class of Kleinian groups. In this paper we review such uniformizations by using
noded Fuchsian groups, which are (geometric) limits of quasiconformal deformations of Fuchsian groups, and which provide uniformizations of stable Riemann orbifolds. These uniformizations permit to obtain a compactification of the Hurwitz spaces with a complex orbifold structure, these being quotients of the augmented Teichm\"uller space of $G$ by a suitable finite index subgroup of its modular group.

\end{abstract}

\section{Introduction}
A {\it $(\gamma,n)$-gonal pair} is a pair $(S,f)$, where $S$ is a closed Riemann surface and $f:S \to R$ is a degree $n \geq 2$ holomorphic branched cover onto a closed Riemann surface $R$ of genus $\gamma$; in this case, we say that $S$ is  {\it $(\gamma,n)$-gonal}. A $(0,n)$-gonal pair is usually called 
an {\it $n$-gonal pair}. In the case that the branched cover $f$ is {\it simple}, that is, each branch value of $f$ has exactly $n-1$ preimages, we talk of 
a {\it simple $(\gamma,n)$-gonal pair}.

There are different notion of equivalences between $(\gamma,n)$-gonal pairs. We are interested in three of them. If 
$(S_{1},f_{1}:S_{1} \to R_{1})$ and $(S_{2},f_{2}:S_{2} \to R_{2})$ are $(\gamma,n)$-gonal pairs, then we say that they are: (i) {\it topologically equivalent} if there are orientation preserving homeomorphisms $\phi:S_{1} \to S_{2}$ and  $\psi:R_{1} \to R_{2}$ so that $f_{2} \circ \phi =\psi \circ  f_{1}$, (ii) {\it twisted isomorphic} if we may assume $\phi$ and $\psi$ to be isomorphisms, and (iii) {\it isomorphic} if $R_{1}=R_{2}$, $\phi$ is an isomorphism and $\psi$ is the identity map. 

Associated to a $(\gamma,n)$-gonal pair $(S,f)$ are the Hurwitz spaces ${\mathcal H}_{0}(S,f)$ and ${\mathcal H}(S,f)$, consisting respectively of the isomorphic classes and of the twisted isomorphic classes of the $(\gamma,n)$-gonal pairs which are topologically equivalent to $(S,f)$. If $(S',f')$ is topologically equivalent to $(S,f)$, then there are natural bijections between (i) ${\mathcal H}_{0}(S,f)$  and ${\mathcal H}_{0}(S',f')$ and (ii) ${\mathcal H}(S,f)$ and ${\mathcal H}(S',f')$.

In 1873, Clebsch \cite{Clebsch} proved that any two simple $n$-gonal pairs are topologically equivalent, in particular, for $(S,f)$ a simple $n$-gonal pair, the space ${\mathcal H}(S,f)$ is the classical Hurwitz space parametrizing twisted isomorphic classes of  degree $n$ simple covers of $\widehat{\mathbb C}$. 
In 1891, Hurwitz \cite{Hurwitz} proved that ${\mathcal H}(S,f)$ has the structure of a complex manifold of dimension $r-3$, where $S$ has genus $g$ and  $r=2g+2n-2 \geq 3$. So, by Clebsch's result, ${\mathcal H}(S,f)$ is irreducible. As every Riemann surface of genus $g$ admits a degree $n \geq g+1$ simple branched cover (by Riemann-Roch's theorem), Severi \cite{Severi} used the above facts to prove that the moduli space ${\mathcal M}_{g}$, of closed Riemann surfaces of genus $g$, is irreducible. For the case of not necessarily simple $n$-gonal pairs $(S,f)$, Fried \cite{Fried} proved that ${\mathcal H}(S,f)$ also has the structure of a connected complex manifold of dimension $r-3$.

In the general case of an $n$-gonal pair $(S,f)$ (i.e., not necessarily simple), in \cite{Natanzon1,Natanzon2} Natanzon  constructed an uniformizing pair $(\Gamma,G)$, where $G$ is a Fuchsian group isomorphic to a free group of rank $|B_{f}|-1$, where
$B_{f}$ is the set of branched values of $f$ (assuming its cardinality is at least $3$), and $\Gamma$ is a suitable finite index subgroup of it (of index equal to the degree of $f$). Using such type of uniformizations, he was able to obtain that ${\mathcal H}(S,f)$ is homeomorphic to the a quotient ${\mathbb R}^{m}/M$, where $M$ is a certain discrete group (in this case, ${\mathbb R}^{m}$ is the real structure of the Teichm\"uller space of the Riemann sphere punctured at $|B_{f}|$ points).

Let us now consider a $(\gamma,n)$-gonal pair $(S,f)$. In Section \ref{Sec:uniformizationsofpairs} we define for such a pair to be of {\it hyperbolic type} and, in such a case, we associate to it a pair $(\Gamma,G)$, called a {\it uniformization of $(S,f)$}, where $G$ is a Fuchsian group acting on the unit disc ${\mathbb D}$ containing $\Gamma$ as an index $n$ subgroup, so that there are isomorphisms $\phi:S \to {\mathbb D}/\Gamma$ and  $\psi: R \to {\mathbb D}/G$ with $\pi=\psi f \phi^{-1}$ being a branched covering induced by the inclusion of $\Gamma$ inside $G$. The uniformizing pair is uniquely determined up to conjugation by holomorphic automorphisms of ${\mathbb D}$. Using such an uniformization pair, it is possible to obtain the following fact that generalizes Natanzon's above.

\s
\noindent
\begin{theo}\label{main1}
Let $(S,f)$ be a $(\gamma,n)$-gonal pair of hyperbolic type and $(\Gamma,G)$ be a unifomizing pair of it. Then 
(i) ${\mathcal H}_{0}(S,f)$ is isomorphic to the Teichm\"uller space $T({\mathbb D},G)$ of the Riemann orbifold ${\mathbb D}/G$, this being a finite dimensional simply-connected complex manifold, 
and (ii) ${\mathcal H}(S,f)$ is a complex orbifold, being a quotient $T({\mathbb D},G)/M(\Gamma,G)$, where $M(\Gamma,G)$ is a suitable finite index subgroup of its holomorphic automorphisms. 
\end{theo}

\s

In \cite{BCI,BC,GWW} there is some study of $(\gamma,n)$-gonal pairs and their groups of automorphisms when $n$ is prime integer and the $n$-gonal map is a regular branched cover.

In the literature, compactifications of Hurwitz spaces of simple $n$-gonal pairs has been obtained by adding the so called stable $n$-gonal pairs (also called admissible ones by some authors; see for instance \cite{H-M}), which are certain geometric degenerations of simple $n$-gonal pairs. Next, we proceed to recall such king of degenerations in the more general case (i.e., it might be either $\gamma \geq 0$ or non-simple situation).

Let $(S,f:S \to R)$ be a fixed $(\gamma,n)$-gonal pair of hyperbolic type. If $\gamma=0$, then we also assume  that its branch value set  $B_{f}$ has cardinality at least $4$; otherwise there is no possible degeneration to be done (if $B_{f}$ has cardinality $3$, then $(S,f)$ is a Belyi pair which is definable over $\overline{\mathbb Q}$ as a consequence of Belyi's theorem \cite{Belyi}). Let us consider a collection ${\mathcal F}$ of pairwise disjoint simple loops $\gamma_{1}, \ldots, \gamma_{s} \subset R-B_{f}$ so that the Euler characteristic of each connected component of $R-(B_{f} \cup \gamma_{1}\cup \cdots \cup \gamma_{s})$ is negative and none of the components of $R-(\gamma_{1} \cup \cdots \cup \gamma_{s})$ is a disc with only two branch values, both with branch order equal to $2$.  Let ${\mathcal G} \subset S$ be the collection of (necessarily simple) loops  obtained by lifting those in ${\mathcal F}$ by $f$. Next, we proceed to identify all points belonging to the same loop in ${\mathcal G}$ to obtain a stable surface $S^{*}$ of the same genus as $S$. Similarly, by doing the same procedure to the loops of ${\mathcal F}$, we obtain an stable genus $\gamma$ orbifold $R^{*}$. The map $f$ induces a continuous map $f^{*}:S^{*} \to R^{*}$ of degree $n$. We call such a pair $(S^{*},f^{*})$ a {\it topological stable $(\gamma,n)$-gonal pair}.  Now, if we provide analytically finite Riemann orbifold structures to each of the components of $R^{*}$ (that is, to the complement of its nodes), then we may lift these Riemann orbifold structures under $f^{*}$ to obtain an analytically finite Riemann orbifold structure on each component of $S^{*}$ minus their nodes. In this way, $S^{*}$ and $R^{*}$ will now carry stable Riemann orbifold structures so that $f^{*}$ still continuous of degree $n$ and its restriction to each connected component of $S^{*}$ minus its nodes is meromorphic. The resulting pairs will be called {\it stable $(\gamma,n)$-gonal pairs modelled by $(S^{*},f^{*})$}. Let us observe that the stable Riemann orbifold structures on $(S^{*},f^{*})$ depends on the spaces of orbifold structures on each of the connected components of $R^{*}$; in fact, the product space of the Teichm\"uller spaces of these components provide a parametrization of the space of stable Riemann orbifold structures of $(S^{*},f^{*})$. In particular, if each of components of $R^{*}$ are spheres with only $3$ marked points, then the structure is unique.

In Section \ref{Sec:nodedpairs} we provide uniformizations of stable $(\gamma,n)$-gonal pairs using pairs of certain noded Fuchsian groups.  This uniformization permits to obtain the following fact.

\s
\noindent
\begin{theo}\label{main2}
Let $(S,f)$ be a $(\gamma,n)$-gonal pair of hyperbolic type and let $(\Gamma,G)$ be an uniformizing pair of it. Then 
\begin{enumerate}
\item[(i)] The augmented Teichm\"uller psce $NT({\mathbb D},G)$ provides a model for the partial closure of ${\mathcal H}_{0}(S,f)$  
obtained by adding the corresponding equivalence classes of stable $(\gamma,n)$-gonal pairs.
\item[(ii)] The space parametrizing twisted isomorphic classes of stable $(\gamma,n)$-gonal pairs can be identified with the quotient $NT({\mathbb D},G)/M(\Gamma,G)$, which has the structure of a compact complex orbifold.
\end{enumerate}
\end{theo}

\s

The above provides a Kleinian groups description being in a parallel point of view as the description provided in \cite{Co-Se} for the case $\gamma=0$. 

An interesting situation is provided for genus zero $n$-gonal pairs as the above provides a description of degenerations of rational maps in terms of Kleinian groups which is somehow related to part of the work done in \cite{Arfeux}.

\s
\noindent
\begin{coro}
Let $R \in {\mathbb C}(z)$ be a rational map of degree $d \geq 2$ and let $B=\{p_{1},\ldots,p_{r}\}$ be its locus of branch values. For each $p_{j}$ let $n_{j}$ be the minimum common multiple of the local degrees of $R$ at its preimages. Assume that $n_{1}^{-1}+\cdots+n_{r}^{-1}<r-2$.
Let $G$ be a Fuchsian group, acting on the unit disc ${\mathbb D}$, uniformizing the orbifold of genus zero whose conical points set is $B$ and the conical order of $p_{j}$ is $n_{j}$. Then the space of isomorphic classes of rational maps topologically equivalent to $R$ is isomorphic to the Teichm\"uller space of $G$, this being a simply-connected complex manifold of dimension $r-3$. Its partial closure obtained by adding isomorphic classes of geometrical degenerations of it is isomorphic to the augmented Teichm\"uller space of $G$. Similarly, the corresponding space of twisted isomorphic classes is a complex orbifold of dimension $r-3$ and its closure obtained by adding the twisted classes of its degenerations is isomorphic to a compact complex orbifold of same dimension.

\end{coro}

\s

This paper is organized as follows. In Section \ref{Sec:prelim1} we review some definitions and basic facts on Kleinian groups, in particular of noded Fuchsian groups and some of its properties previously obtained in \cite{Hidalgo:nodedSchottky, Hidalgo:nodedFuchsian}.  In Section \ref{Sec:prelim2} we recall the definition and some facts on the quasiconformal deformation of Kleinian groups, in particular, the Teichm\"uller space of a finitely generated Fuchsian group. In Section \ref{Sec:NodedBeltrami} we recall the concept of noded Beltrami coefficients of Kleinian groups from \cite{Hidalgo-Vasiliev} which permits to construct a partial closure of the quasiconformal deformation space of a Kleinian group. We will be mainly interested, for the purpose of this article, on the case of Fuchsian groups, but we provide the general point of view. In the final two sections we use the previous facts to describe, in terms of Fuchsian groups and noded Fuchsian groups, the unifomizations of $(\gamma,n)$-gonal and stable $(\gamma,n)$-gonal pairs, respectively.

\section{Some preliminaries on Kleinian groups}\label{Sec:prelim1}

\subsection{Analytically finite Riemann orbifolds}
An {\it analytically finite Riemann orbifold} is given by a closed Riemann surface $S$ of genus $g \geq 0$ (the underlying Riemann surface structure of the Riemann orbifold) together with a finite collection of conical points $x_{1},...,x_{r} \in S$ of orders $2 \leq m_{1} \leq m_{2} \leq \cdots \leq m_{r} \leq \infty$, respectively. If $x_{j}$ has cone order $\infty$, then this means that it represents a puncture of the orbifold. The {\it signature} of the orbifold is defined by the tuple $(g;m_{1},...,m_{r})$ and it is called {\it hyperbolic} if  $2g-2+\sum_{i=1}^{r}(1-m_{i}^{-1})>0$.

If (i) $r=3$ and $g= 0$, then the orbifold is called a {\it triangular orbifold}, (ii) if $r=0$, then the orbifold is a closed Riemann surface, and (iii) if $r>0$ and $m_{j}=\infty$ for all $j$, then it is an analytically finite punctured Riemann surface (this is $S$ minus all the points $x_{j}$).

\subsection{Noded Riemann surfaces/orbifolds}
Let us consider a countable collection $\{{\mathcal O}_{j}\}_{j \in J}$ of Riemann orbifolds, that is, each ${\mathcal O}_{j}$ consists of a Riemann surface $S_{j}$  together a discrete collection of cone points $B_{j}=\{p_{ji}\} \subset S_{j}$ and integer values $n_{ji} \geq 2$ (the cone orders of the points $p_{ji}$). Let us consider a discrete collection of points ${\mathcal E} \subset \bigcup_{j \in J} (S_{j}-B_{j})$ and an order two bijective map $T:{\mathcal E} \to {\mathcal E}$.

 We proceed to identify the points $q$ and $T(q)$, for each $q \in {\mathcal E}$, to obtain a space $X$; called a {\it noded Riemann orbifold}. The points obtained by the identification of the points $q$ and $T(q)$, for $q \in {\mathcal E}$, are called (i) {\it nodes} of $X$ if $T(q)=q$ and (ii) {\it phantom nodes} if $T(q)=q$. We denote the set of nodes and phantom nodes of $X$ by $N(X)$. Observe that
the points in $N(X)$ correspond to punctures on each connected component of $X-N(X)$ and that these components are given by $S_{j}-({\mathcal E}  \cap S_{j})$. The set of cone points of $X$ is given by $B(X):=\bigcup_{j} B_{j}$.

In the case that every $B_{j}=\emptyset$ (that is, ${\mathcal O}_{j}$ is just a Riemann surface), we say that $X$ is a  {\it noded Riemann surface}. 

An isomorphism between noded Riemann orbifolds is a homeomorphism that send cone points to cone points (respecting their orders) which, restricted to the complement of the nodes, is analytic. If there is an isomorphism between two noded Riemann orbifolds/surfaces, 
then we say that they are isomorphic or conformally equivalent. This, in particular, permits to talk on automorphisms of noded Riemann orbifolds/surfaces.

A noded Riemann surface which is homeomorphic to the space obtained by pinching a non-trivial simple loop on a torus is called a {\it stable Riemann surface of genus one}.  A noded Riemann surface which is homeomorphic to the space obtained by the process of pinching a family (necessarily finite) of  
disjoint simple closed hyperbolic geodesics on a closed Riemann surface of genus $g \geq 2$ is called a {\it stable Riemann surface of genus $g$}.
Isomorphic classes of stable Riemann surfaces are the extra points Mumford needed to add to the moduli space ${\mathcal M}_{g}$ of closed Riemann surfaces of genus $g$ to provide the Deligne-Mumford's compactification $\overline{\mathcal M}_{g}$.

Let us observe that the space obtained as the quotient $X/H$, where $X$ is a noded Riemann surface and $H$ is a (finite) group of automorphisms of $X$
is an example of a  noded Riemann orbifold. In the case that $X$ is a stable Riemann surface, then $X/H$ is also called a {\it stable Riemann orbifold}.

\subsection{Kleinian groups}
A {\it Kleinian group} is just a discrete subgroup $G$ of ${\rm PSL}_{2}({\mathbb C})$ (seen as the group of M\"obius transformations acting on $\widehat{\mathbb C}$). Generalities on Kleinian/Fuchsian groups can be found, for instance, in the books \cite{Beardon,M1}. In this section we recall some of the basics we will need in the rest of this paper.

\subsubsection{The region of discontinuity}
The  {\it region of discontinuity} of a Kleinian group $G$ is the (might be empty) open set $\Omega(G)$ of points over which $G$ acts properly discontinuous; the complement $\Lambda(G)=\widehat{\mathbb C}-\Omega(G)$ is its {\it limit set}.  If $\Omega(G) \neq \emptyset$, then the quotient space $\Omega(G)/G$ is a union of Riemann orbifolds. By Ahlfor's finiteness theorem \cite{Ahlfors:finitud}, if $G$ is finitely generated, then such a quotient consists of a finite number of analytically finite Riemann orbifolds. 

Let $\delta \in G$ be a loxodromic transformation. We say that $\delta$ is {\it primitive} if it is not a nontrivial positive power of another loxodromic transformation in $G$. We say that $\delta$ is {\it simple loxodromic} if there is a simple arc on $\Omega(G)$ which is invariant under $\delta$, we call it an {\it axis} of $\delta$, whose projection on $\Omega(G)/G$ is a simple loop or a simple arc connecting two conical points of order $2$. 

\subsubsection{The extended region of discontinuity}
Let $G$ be a finitely Kleinian group with non-empty region of discontinuity.

A parabolic transformation $\eta \in G$, with fixed point $p$, is called {\it double-cusped}, if (i) any parabolic element of $G$ commuting with $\eta$ belongs to the cyclic group $\langle \eta \rangle$, and (ii) there are two tangent  open discs at $p$ in $\Omega(G)$ whose union is invariant under the stabilizer of $p$ in $G$.

The  {\it extended region of discontinuity} of $G$ is defined as $\Omega(G)^{ext}=\Omega(G) \cup P(G)$, where $P(G)$ is the set of fixed points of
the double-parabolic elements of $G$.

On $\Omega(G)^{ext}$ we consider its cuspidal topology; the topology
generated by the usual open sets in $\Omega(G)$ and the sets of the form 
$D_{1} \cup D_{2} \cup \{p\}$, where $p \in P(G)$, and $D_{1}, D_{2} \subset
\Omega(G)$ are round discs tangent at $p$. Observe that if $G$ has no parabolic transformations,
then the extended region of discontinuity coincides with its region of
discontinuity, that is, $\Omega(G)=\Omega(G)^{ext}$. 

In the cuspidal topology, the group $G$ acts as a group of homeomorphisms on $\Omega(G)^{ext}$,
keeping invariant each $\Omega(G)$ and $P(G)$, and its restriction to
$\Omega(G)$ being by holomorphic automorphisms. 

By Selberg's lemma, $G$ contains a torsion-free finite index normal subgroup $K$. The finite index condition asserts that $\Omega(K)=\Omega(G)$ and $P(K)=P(G)$, in particular, $\Omega^{ext}(K)=\Omega^{ext}(G)$. The quotient $\Omega(G)^{ext}/K$ is a noded Riemann surface whose nodes corresponds one-to-one to the $K-$equivalence classes of the parabolic fixed points in $P(G)$ (those having an involution in the stabilizer produce the so called phantom nodes).  By Ahlfors' finiteness theorem,  
the number of components of $\Omega(G)/K=\Omega(G)^{ext}/K -N(\Omega(G)^{ext}/K)$ 
is finite, each one an analytically finite Riemann surface. It follows that 
$\Omega(G)^{ext}/G$, with the quotient cuspidal topology, is a noded Riemann orbifold, and it contains the orbifold $\Omega(G)/G$ as a dense open subset. The points in $P(G)/G$ are the nodes (and phantom nodes) of $\Omega(G)^{ext}/G$.

\subsection{Fuchsian groups}
Let ${\mathbb D}$ be the unit disc in the complex plane and  let ${\rm Aut}({\mathbb D})$ be its group of conformal automorphisms (this being a subgroup of ${\rm PSL}_{2}({\mathbb C})$). 

A finitely generated Kleinian group $F$ being a subgroup of ${\rm Aut}({\mathbb D})$ is called a {\it Fuchsian group}. The Fuchsian group $F$ is called of the {\it first kind} if its limit set is all of the unit circle (so is region of discontinuity consists of two discs); otherwise, it is called of the {\it second kind} (so its region of discontinuity is connected).

As a consequence of Selberg's lemma \cite{Beardon}, $F$ has a torsion-free normal subgroup $K$ of finite index; 
this is again a finitely generated Fuchsian group of the same kind as $F$.
If $F$ is of the second kind without parabolic elements, then $K$ is a Schottky group of some finite rank $g \geq 0$ and if $F$ is of the first kind without parabolic elements, then $K$ is a co-compact Fuchsian group uniformizing a closed Riemann surface of some genus $g \geq 2$ (in this case we say that $K$ is a $\pi^{g}$ group).

By classical work of Fricke and Klein, if the Fuchsian group $F$ is of the first kind and without parabolic elements, then it has a presentation of the form
$$F=\langle \alpha_{1},\ldots,\alpha_{g},\beta_{1},\ldots,\beta_{g},\delta_{1},\ldots,\delta_{r}:\prod_{j=1}^{g}[\alpha_{j},\beta_{j}] \prod_{i=1}^{r}\delta_{i}=\delta_{1}^{m_{1}}=\cdots=\delta_{r}^{m_{r}}=1\rangle,$$
where $[\alpha_{j},\beta_{j}]=\alpha_{j}\beta_{j}\alpha_{j}^{-1}\beta_{j}^{-1}$, $m_{i} \geq 2$ are integers so that
$2g-2+\sum_{i=1}^{r}(1-m_{i}^{-1})>0$. In this case, ${\mathbb D}/F$ is a compact hyperbolic Riemann orbifold of signature 
$(g;m_{1},\ldots,m_{r})$; which is also called the signature of $F$.  If $r=0$, then ${\mathbb D}/F$ is a closed Riemann surface of genus $g$ and its signature is denoted by $(g;-)$. 

As a consequence of the uniformization theorem, every compact hyperbolic orbifold ${\mathcal O}$ is isomorphic to ${\mathbb D}/F$ for some Fuchsian group $F$.

\subsection{Noded Fuchsian groups}
A Kleinian group is called a {\it noded Fuchsian group} \cite{Hidalgo:nodedFuchsian}
if it is a geometrically finite and isomorphic to some Fuchsian group of the first kind.

Torsion-free noded Fuchsian groups come in two flavors \cite{Hidalgo:nodedSchottky,Hidalgo:nodedFuchsian}:  (i) {\it noded Schottky groups of rank $g\geq 0$} (isomorphic to free groups of rank $g\geq 0$) and (ii) {\it noded $\pi^{g}$ groups} (isomorphic to the fundamental group of a closed orientable surface of genus $g \geq 2$). So, as a consequence of Selberg's lemma, every noded Fuchsian group has a finite index normal subgroup being either a noded Schottky group or a noded $\pi^{g}$ group.

\s
\noindent
\begin{theo}[\cite{Hidalgo:nodedFuchsian}]
Noded Fuchsian groups have non-empty region of discontinuity.
\end{theo}

\s

If $G$ is a noded Fuchsian group, then, as $G$ is isomorphic to some Fuchsian group, it cannot have rank two parabolic subgroups and all parabolics are double-cusped. As already noted above, $G$ contains a torsion-free finite index normal subgroup $K$, this being either a noded Schottky group or a noded $\pi^{g}$ group. The finite index condition asserts that $\Omega(K)=\Omega(G)$ and $P(K)=P(G)$, in particular, $\Omega^{ext}(K)=\Omega^{ext}(G)$. 
 Below, in Sections \ref{Sec:nrt} and \ref{Sec:nsut}, we will see that noded Riemann surface $\Omega(G)^{ext}/K$ consists of either: (i) two stable Riemann surfaces of genus $g$, if $K$ is noded $\pi^{g}$-group, or (ii) one stable Riemann surface of genus $g$, if $K$ is a noded Schottky group of rank $g$). It follows that $\Omega(G)^{ext}/G$, with the quotient cuspidal topology, is a stable Riemann orbifold, and it contains the Riemann orbifold $\Omega(G)/G$ as a dense open subset. The points in $P(G)/G$ are the nodes of $\Omega(G)^{ext}/G$ (and there are no phantom nodes).  

As an example, the cyclic group $G=\langle \gamma(z)=z+1 \rangle$ is a noded Fuchsian group (in fact a noded Schottky group of rank one). In this case, $\Omega(G)^{ext}=\widehat{\mathbb C}$, $\Omega(G)={\mathbb C}$ and $\Omega(G)^{ext}/G$ is a stable Riemann surface of genus one (its node being the projection of $\infty$). 

\subsection{Noded retrosection theorem}\label{Sec:nrt}
Koebe's retrosection theorem \cite{Koebe} asserts that every closed Riemann surface of genus $g$ can be uniformized by a Schottky group of rank $g$.
A retrosection theorem with nodes hold and it can be stated as follows.

\s
\noindent
\begin{theo}[Noded retrosection theorem \cite{Hidalgo:nodedSchottky,Hidalgo:nodedFuchsian}]
\mbox{}
\begin{enumerate}
\item If $G$ is a noded Schottky group of rank $g$, then $\Omega(G)^{ext}$ is connected and
$\Omega(G)^{ext}/G$ is a stable Riemann surface of genus $g$.

\item If $S$ is a stable Riemann surface of 
genus $g$, then there is a noded Schottky group $G$ of 
genus $g$ such that $\Omega(G)^{ext}/G$ is conformally equivalent to $S$.
\end{enumerate}
\end{theo}

\s
\noindent
\begin{rema}
The region of discontinuity of a Schottky group is always connected, but that of a noded Schottky group is not in general; connectivity only holds for its extended region of discontinuity.

\end{rema}

\s
\noindent
\begin{coro}
If $G$ is a noded Fuchsian group containing a noded Schottky group as a finite index normal subgroup, then $\Omega(G)^{ext}$ is connected and $\Omega(G)^{ext}/G$ consists of an stable Riemann orbifold.
\end{coro}

\s

\subsection{Simultaneous uniformization theorem with nodes}\label{Sec:nsut}
If $G$ is a torsion free purely loxodromic quasifuchsian group (i.e., a quasiconformal deformation of a torsion free Fuchsian group of the first kind without parabolics), then its region of discontinuity consists of two topological discs, say $D_{1}$ and $D_{2}$, and the quotients $D_{1}/G$ and $D_{2}/G$ are closed Riemann surfaces of the same genus $g \geq 2$.  Bers' simultaneous uniformization theorem \cite{Bers} asserts that, given any two closed Riemann surfaces $S_{1}$ and $S_{2}$ of the same genus $g \geq 2$, then it is possible to find $G$ as above so that we may assume that $S_{j}$ is isomorphic to $D_{j}/G$, for $j=1,2$.

Let $G$ be now a noded $\pi^{g}$ group ($g \geq 2$). If $G$ is purely loxodromic, then Maskit \cite{M3} has shown that $G$ is 
in fact a quasifuchsian group and we are as above. So, we are left to consider the case when $G$ contains parabolic transformations.

In \cite{Hidalgo:nodedFuchsian,M3} there provided examples of noded $\pi^{g}$ groups (with parabolic transformations) which are not quasifuchsian ones (the example in \cite{M3} is a B-group \cite{M1} and that in \cite{Hidalgo:nodedFuchsian} is a group without invariant components in its region of discontinuity).  But, if consider the extended region of discontinuity, then the situation is similar to the quasifuchsian case, as seen in the following.

\s 
\noindent
\begin{theo}[Simultaneous uniformization theorem with nodes \cite{Hidalgo:nodedFuchsian, Kra-Maskit}] 
\mbox{}
\begin{enumerate}
\item If $g \geq 2$ and $G$ is a noded $\pi^{g}$ group, then $\Omega(G)^{ext}$ consists
of exactly two simply-connected invariant components, and $\Omega(G)^{ext}/G$ consists of exactly
two stable Riemann surface of genus $g$.

\item If $S$ and $R$ are two stable Riemann surfaces of genus $g \geq 2$, then there is a 
noded $\pi^{g}$ group $G$ such that $\Omega(G)^{ext}/G$ is isomorphic to $S \cup R$.

\end{enumerate}
\end{theo}

\s

As a consequence, the same fact holds for those noded Fuchsian groups containing noded $\pi^{g}$ groups.

\s
\noindent
\begin{coro}
Let $G$ be a noded Fuchsian group, containing a noded $\pi^{g}$ group as a finite index normal subgroup. Then $\Omega(G)^{ext}$ consists of two simply-connected invariant components, and $\Omega(G)^{ext}/G$ consists of exactly two stable Riemann orbifolds.
\end{coro}

\s

The above result was proved by Abikoff in \cite{Ab1} for the case
of cusps, that is, for regular B-groups.

\section{Some preliminaries on Teichm\"uller spaces of Kleinian groups}\label{Sec:prelim2}
In this section, $G$ will be a finitely generated Kleinian group, with non-empty region of discontinuity $\Omega(G)$, and $\Delta \subset \Omega(G)$ will be a non-empty $G$-invariant collection of components of $\Omega(G)$, that is, every $\gamma \in G$ permutes the connected components of $\Delta$.

We proceed to recall the quasiconformal deformation space $T(\Delta,G)$ of $G$ supported at $\Delta$. In the case that $G$ is Fuchsian group of the first kind acting on the unite disc ${\mathbb D}$, then $T({\mathbb D},G)$ is a model for the Teichm\"uller space of the Riemann orbifold ${\mathbb D}/G$.

\subsection{Quasiconformal homeomorphisms}
Let us consider the Banach space $L^{\infty}(\widehat{\mathbb C})$ of measurable maps $\mu:\widehat{\mathbb C} \to {\mathbb C}$,
with the essential supreme norm $\|\cdot\|_{\infty}$. We denote its unit ball as 
$L^{\infty}_{1}(\widehat{\mathbb C})$.

An orientation preserving homeomorphism $w:\widehat{\mathbb C} \to
\widehat{\mathbb C}$ is called a {\it quasiconformal homeomorphism} if there is some 
 $\mu \in L^{\infty}_{1}(\widehat{\mathbb C})$
(called a {\it complex dilatation} of $w$) such that $w$ has distributional partial
derivatives $\partial w$, $\overline{\partial} w$ in
$L^{2}_{loc}(\widehat{\mathbb C})$ satisfying the Beltrami equation
$$\overline{\partial}w(z)=\mu(z)\partial w(z), \quad \mbox{ a.e. $z \in
\widehat{\mathbb C},$ }$$
where $L^{2}_{loc}(\widehat{\mathbb C})$ means $L^{2}$ on compacts in $\widehat{\mathbb C}$.

The existence and uniqueness of quasiconformal homeomorphisms is due to Morrey \cite{Morrey}.

\s

\begin{theo}[Measurable Riemann mapping theorem \cite{A-B,Morrey}]\label{A-B}
If $\mu \in L^{\infty}_{1}(\widehat{\mathbb C})$, then there exists, and it is  unique,
a quasiconformal homeomorphism $w_{\mu}:\widehat{\mathbb C} \to \widehat{\mathbb C}$, with complex dilation $\mu$, and fixing $\infty$, $0$ and $1$. Moreover, if $\mu$ vary continuously or real-analytically or holomorphically and $z_0$ is a fixed point, then $w_{\mu}(z_0)$ varies also in the same way.
\end{theo}

\s
\subsection{Beltrami coefficients for $G$ supported at $\Delta$}
The discreteness property of $G$ asserts that 
$$L^{\infty}(\Delta,G)=
\left\{\mu \in L^{\infty}(\widehat{\mathbb C}): \; \mu(\gamma(z)) \overline{\gamma'(z)} = \mu(z) \gamma'(z), \;
{\rm a.e.}\; \Delta,\; \mbox{for all} \; \gamma \in G, \; \mu(z)=0,\; z \in \Delta^{c}\right\}$$
is a closed subspace of $L^{\infty}(\widehat{\mathbb C})$; so it is a Banach space with essential supreme norm $\|\cdot\|_{\infty}$. 
Let us denote its unit ball as 
$$L^{\infty}_{1}(\Delta,G)=L^{\infty}(\Delta,G) \cap L^{\infty}_{1}(\widehat{\mathbb C})=\left\{\mu \in L^{\infty}(\Delta,G); \;
\|\mu\|_{\infty} < 1\right\}.$$

The measurable functions in $L^{\infty}_{1}(\Delta,G)$ are called the 
{\it Beltrami coefficients for $G$ supported in $\Delta$}. 

The following lemma is a 
classical result and its proof can be found, for instance, in \cite{L}.

\s
\noindent
\begin{lemm} \label{lema1}
Let  $\mu \in L^{\infty}_{1}(\Delta,G)$ and let
$w:\widehat{\mathbb C} \to \widehat{\mathbb C}$
be a quasiconformal homeomorphism with complex dilatation $\mu$.
Then $w G w^{-1}$ is a Kleinian group with the
region of discontinuity $w(\Omega(G))$. Moreover, $w(\Delta)$ is a $W G w^{-1}$-invariant collection of components.
\end{lemm}

\s
\subsection{Teichm\"uller and moduli spaces}
Observe that, by the measurable Riemann mapping theorem, if $\mu \in L^{\infty}_{1}(\Delta,G)$ and $w_{1}, w_{2}$ are quasiconformal homeomorphisms, both with complex dilation $\mu$, then there is a M\"obius transformation $A \in {\rm PSL}_{2}({\mathbb C})$ so that $w_{2}=A w_{1}$.

Let  $\mu,\nu \in L^{\infty}_{1}(\Delta,G)$ be two Beltrami coefficients for $G$ supported in $\Delta$. If $w_{\mu}$ (respectively $w_{\nu}$) is a 
quasiconformal homeomorphism with complex dilation $\mu$ (respectively $\nu$), then (by Lemma \ref{lema1}) there is a natural isomorphism $\theta_{\mu}:G \to w_{\mu} G w_{\mu}^{-1}$ (respectively $\theta_{\nu}:G \to w_{\nu} G w_{\nu}^{-1}$). We say that $\mu$ and $\nu$ are {\it Teichm\"uller equivalent} (respectively, {\it isomorphic})
if there is $A \in {\rm PSL}_{2}({\mathbb C})$ with 
$\theta_{\mu}(\gamma)=A \theta_{\nu}(\gamma) A^{-1}$, for all $\gamma \in G$ (respectively, $\theta_{\mu}(G)=A \theta_{\nu}(G) A^{-1}$).

The space $T(\Delta,G)$ (respectively,  ${\mathcal M}(\Delta,G)$) of Teichm\"uller (respectively, isomorphic) equivalence classes of Beltrami coefficients for $G$ supported in $\Delta$ is called the {\it Teichm\"uller space} (respectively,  {\it moduli space}) {\it of $G$  supported in $\Delta$}.

The {\it modular group of $G$ supported at $\Delta$} is the group $M(\Delta,G)$ given by the isotopy classes of quasiconformal homeomorphisms $w:\widehat{\mathbb C} \to \widehat{\mathbb C}$, with complex dilation in $L^{\infty}_{1}(\Delta,G)$, so that $w G w^{-1}=G$. There is the natural action
$$M(\Delta,G) \times T(\Delta,G) \to T(\Delta,G): ([w],[\mu]) \mapsto [\nu],$$
where $\nu \in L^{\infty}_{1}(\Delta,G)$ is complex dilation of the quasiconformal homeomorphism $w_{\mu} w^{-1}$.

It is well known that $T(\Delta,G)$ is a finite dimensional complex manifold (simply-connected in the case that $\Delta$ is connected and simply-connected), that $M(\Delta,G)$ acts as a discrete group of its holomorphic automorphisms and that ${\mathcal M}(\Delta,G)=T(\Delta,G)/M(\Delta,G)$ is a complex orbifold of same dimension as $T(\Delta,G)$ (generalities on this can be found, for instance, in \cite{Bers2,Kra,M5,N}).

\subsection{The Fuchsian case} \label{Sec:Fuchsian}
Let us assume that $G$ is a Fuchsian group of the first kind 
acting on the unit disc ${\mathbb D}$; so the Riemann orbifold ${\mathbb D}/G$ has signature $(\gamma;m_{1},\ldots,m_{r})$. In this case $T({\mathbb D},G)$ is a simply-connected complex manifold of dimension $3\gamma-3+r$ (see, for instance, the book \cite{N}) and, moreover, this space is a model for the Teichm\"uller space of the Riemann orbifold ${\mathbb D}/G$.

If $\Gamma$ is a finite index subgroup of $G$, then the natural inclusion $L^{\infty}_{1}({\mathbb D},G) \subset L^{\infty}_{1}({\mathbb D},\Gamma)$ induces a holomorphic embedding $T({\mathbb D},G) \subset T({\mathbb D},\Gamma)$. Also, in this case, the subgroup $M(\Gamma,G)$ of $M({\mathbb D},G)$, formed by those isotopy classes of quasiconformal homeomorphisms for which  $w \Gamma w^{-1}=\Gamma$, has finite index. So the complex orbifold $T({\mathbb D},G)/M(\Gamma,G)$ provides a finite degree branched cover of moduli space ${\mathcal M}({\mathbb D},G)$.

\s
\noindent
\begin{rema}\label{obs2}
Let $G$ a finitely generated Kleinian group and let $\Delta$ a $G$-invariant collection of connected components of $\Omega(G) \neq \emptyset$.
Let $\Delta_{1},\ldots,\Delta_{r}$ be a maximal collection of non-$G$-conjugate components of $\Delta$. Let $G_{j}$ be the $G$-stabilizer of $\Delta_{j}$ and  
let $G(\Delta_{j})$ be the union of all components of $\Delta$ which are $G$-conjugate to $\Delta_{j}$. As a consequence of results of Kra \cite{Kra} and Sullivan \cite{Sullivan}, it was observed by Kra and Maskit in \cite{K-M} that 
$$T(\Delta,G)=T(G,G(\Delta_{1})) \times \cdots \times T(G,G(\Delta_{r}))$$
and that its universal cover space is 
$$T({\mathbb D},F_{1})\times \cdots \times T({\mathbb D},F_{r}),$$ 
where $F_{j}$ is a Fuchsian group acting on ${\mathbb D}$ so that ${\mathbb D}/F_{j} \cong \Delta_{j}/G_{j}$. This in particular asserts that that the statement at the end of Section \ref{Sec:Fuchsian} still valid for $G$.
\end{rema}

\s

\section{Noded quasiconformal deformation spaces of Kleinian groups}\label{Sec:NodedBeltrami}
If $G$ is a Fuchsian group of the first kind acting on the unit disc ${\mathbb D}$, many different compactifications of $T({\mathbb D},G)$ have been provided. For instance, Bers' compactification \cite{Bers3} is obtained by holomorphically embedding $T({\mathbb D},G)$ into the space of quadratic holomorphic forms on ${\mathbb D}/G$ (this being a finite dimensional complex vector space) and taking its closure in there, and Thurston's compactification \cite{Thurston1988} is obtained by taking hyperbolic lengths at simple closed geodesics, which provides a holomorphic embedding into an infinite-dimensional projective space. On Ber's compactifications it is no possible to extend continuously the action of the corresponding modular group. There is not a natural relation between these compactifications and Deligne-Mumford's compactification. A partial closure of $T({\mathbb D},G)$, called the {\it Augmented Teichm\"uller space} $\widehat{T}({\mathbb D},G)$, was constructed by Bers \cite{Bers3} (see also W. Abikoff in \cite{Abikoff}). This space is a non-compact Hausdorff space over which the modular group $M({\mathbb D},G)$ extends continuously and so that $\widehat{T}({\mathbb D},G)/M({\mathbb D},G)$ coincides with Deligne-Mumford's compactification. The added points to $\widehat{T}({\mathbb D},G)$ are certain regular b-groups (geometrically finite Kleinian groups with an invariant simply-connected connected of its region of discontinuity \cite{M1})

If $\Delta$ is a collection of $G$-invariant components of $\Omega(G)$, where $G$ is a finitely generated Kleinian group, then
in \cite{Hidalgo-Vasiliev} we have constructed a partial closure $NT(\Delta,G)$ of the Teichm\"uller space $T(\Delta,G)$ so that the boundary points correspond in a natural way to the boundary points in the Deligne-Mumford compactification of moduli space of $\Delta/G$. In the particular case that $G$ is Fuchsian of the first kind and $\Delta={\mathbb D}$, the partial closure $NT({\mathbb D},G)$ coincides with the augmented Teichm\"uller space. The extra points we add to $T(\Delta,G)$ produce, in terms of Kleinian groups, stable Riemann orbifolds  and they correspond to deformations of the group $G$ by the process of approximation of double-cusped parabolic elements of $G$ by certain {\it primitive simple loxodromic} ones  (see the works of Keen, Series and Maskit in \cite{KMS}, and of Maskit in \cite{M2,M4}). The deformation is produced by some boundary points of $L^{\infty}_{1}(G,\Delta)$, called {\it noded Beltrami differentials for $G$}. At the level of the Riemann orbifold $\Delta/G$ this means that we permit certain pairwise disjoint simple closed geodesics and maybe some simple geodesic arcs connecting conical values of order $2$ to degenerate to points in order to produce a finite collection of noded Riemann orbifolds.  The loops and arcs which we consider in the degeneration process are the projections of appropriately chosen axes of the primitive simple loxodromic elements of the group that approach  doubly-cusped parabolic transformations. 

In this section we will assume that $G$ is a finitely generated Kleinian group, with non-empty region of discontinuity, and that $\Delta$ a non-empty $G$-invariant collection of connected components of its region of discontinuity, and we proceed to recall the construction of $NT(\Delta,G)$ as done in \cite{Hidalgo-Vasiliev}.

\subsection{Region of discontinuity of Beltrami forms}
For each $\mu \in L^{\infty}(\widehat{\mathbb C})$ we define its {\it region of discontinuity} $\Omega(\mu)$  as the set of all points $p \in \widehat{\mathbb C}$ for which there is 
an open neighborhood $U$ of $p \in U$ so that $\|\mu|_{U}\|_{\infty} <1$.
Its complement $\Lambda(\mu)=\widehat{\mathbb C}-\Omega(\mu)$ is the {\it limit set}
of $\mu$. By  the definition, the set $\Omega(\mu)$ is open and $\Lambda(\mu)$ is
compact.

If $\mu \in \overline{L^{\infty}_{1}(\Delta,G)}$, the closure of the unit ball 
$L^{\infty}_{1}(\Delta,G)$ inside the Banch space $L^{\infty}(\Delta,G)$, then
in \cite{Hidalgo-Vasiliev} it was observed that both $\Omega(\mu)$ and $\Lambda(\mu)$ are $G-$invariant.
Also, by the $G$-invariance of $\Omega(\mu)$ and $\Omega(G)$, the open set $\Omega:=\Omega(\mu) \cap \Omega(G)$ is $G$-invariant. Observe that, as $\mu$ is equal to zero on the complement of $\Delta$, all connected components of $\Omega(G)-\Delta$ are necessarily contained in $\Omega$.

\subsection{Noded quasiconformal maps for $(\Delta,G)$}
If $V \subset \widehat{\mathbb C}$ is a non-empty open set, then $L^{2,1}_{loc}(V)$ denotes the complex vector space of maps $w:V \to \widehat{\mathbb C}$ with locally integrable distributional derivatives. 

Let $\mu \in \overline{L^{\infty}_{1}(\Delta,G)}$, where $\Omega(\mu) \neq \emptyset$. An orientation-preserving map $w \in L^{2,1}_{loc}(\Omega(\mu))$ is a {\it noded quasiconformal map for $(\Delta,G)$} with dilatation $\mu$ if the following hold:
\begin{itemize}
\item[(i)] there is a component of $\Omega(\mu)$ homeomorphically mapped by $w$ onto its image;
\item[(ii)] $ \overline{\partial} w(z) = \mu(z) \partial w(z), \,\,\, \mbox{\rm a.e. in $\Omega(\mu)$}$;
\item[(iii)] there is a sequence $\mu_{n} \in L^{\infty}_{1}(\Delta,G)$,
converging to $\mu$ almost everywhere in $\Omega(\mu)$;
\item[(iv)] there is a sequence $w_{n}:\widehat{\mathbb C} \to \widehat{\mathbb C}$ of
quasiconformal homeomorphisms with complex dilatations $\mu_{n}$, converging to $w$
locally uniformly in $\Omega(\mu)$.
\end{itemize}

\s

In the above definition there are many properties to be checked, but the 
following existence result is classical and a proof can be find in \cite{Hidalgo-Vasiliev}. 

\s
\noindent
\begin{prop}\label{Prop2}  
Let $\mu \in \overline{L^{\infty}_{1}(\Delta,G)}$ be such that $\Omega(\mu) \neq \emptyset$, $\Omega_{1}$ be a connected component
of $\Omega(\mu)$ and $x_{1}, x_{2}, x_{3} \in \Omega_{1}$ three different points. Then 
there is a noded quasiconformal map $w:\Omega(\mu) \to \widehat{\mathbb C}$, for $(\Delta,G)$, with dilatation $\mu$,  
fixing the points $x_{1}$, $x_{2}$ and $x_{3}$,
which is a homeomorphism when restricted to $\Omega_{1}$. 
\end{prop}

\s
\noindent
\begin{rema}
Let us observe that the noded quasiconformal map $w$ for $(\Delta,G)$ in Proposition \ref{Prop2} might be constant on some other connected component of $\Omega(\mu)$. 
\end{rema}

\s

\subsection{Noded family of arcs}
A countable collection ${\mathcal F}_{\mu}=\{ \alpha_{1},\alpha_{2},...\}$ of pairwise disjoint simple arcs
(including end points)
is called a {\it noded family of arcs associated with $\mu \in \overline{L^{\infty}_{1}(G,\Delta)}$}, if the
following properties hold:
\begin{itemize}
\item[(1)] $\alpha_{n}^{*} \subset \Delta$, where $\alpha_{n}^{*}$ denotes
$\alpha_{n}$ minus both extremes; 
\item[(2)] the spherical diameter of
$\alpha_{n}$ goes to $0$ as $n$ goes to $\infty$; 
\item[(3)] $\Lambda(\mu)=\overline{\cup_{j=1}^{\infty}\alpha_{j}}$;
\item[(4)] $\Omega(\mu) \subset \widehat{\mathbb C}$ is a dense subset; 
\item[(5)] the group $G_{n}=\{g \in G; g(\alpha_{n})=\alpha_{n}\}$, is either
a cyclic loxodromic group or a ${\mathbb Z}_{2}-$extension of a cyclic 
loxodromic group.
\end{itemize}

\s
\noindent
\begin{rema}
There are exist examples of  $\mu \in \overline{L^{\infty}_{1}(G,\Delta)}$ for which there is no associated noded family of arcs. In \cite{Hidalgo-Vasiliev} there are constructed examples in the positive direction.
\end{rema}

\s

\subsection{Noded Beltrami coefficients for $G$ supported in $\Delta$}
By the stereographic projection, we may see the Riemann sphere as the unit sphere in ${\mathbb R}^{3}$. This allows us to consider the spherical metric and work with the spherical diameter of a subset of $\widehat{\mathbb C}$. 

An element $\mu \in  \overline{L^{\infty}_{1}(\Delta,G)}$ will be called a 
{\it noded Beltrami coefficient for $G$ supported in $\Delta$} if:

\begin{enumerate}
\item[(I)] $\mu$ has a noded family of arcs ${\mathcal F}_{\mu}=\{ \alpha_{1},\alpha_{2},...\}$, and

\item[(II)] there is a continuous map $w:\widehat{\mathbb C} \to \widehat{\mathbb C}$, called a {\it noded quasiconformal deformation of $G$ with complex dilatation $\mu$}, satisfying the following properties:  
\begin{enumerate}
\item[(1)] $w$ is injective in $\widehat{\mathbb C}-{\mathcal F}_{\mu}$;  
\item[(2)] $w:\Omega(\mu) \to \widehat{\mathbb C}$ is a noded quasiconformal map for $(\Delta,G)$ with complex dilatation $\mu$; 
\item[(3)] the restriction of $w$ to each arc $\alpha_{i}$ is a constant 
$p_{i}$, where $p_{i} \neq p_{j}$ for $i \neq j$. 
\end{enumerate}
\end{enumerate}

We denote by $L^{\infty}_{{\rm noded}}(\Delta,G)$ the subset of $\overline{L^{\infty}_{1}(\Delta,G)}$ consisting of the noded Beltrami coefficients for $G$ supported at $\Delta$. Clearly, $L^{\infty}_{1}(\Delta,G) \subset L^{\infty}_{{\rm noded}}(\Delta,G)$.

\s
\noindent
\begin{rema}
 Let $\mu \in L^{\infty}_{{\rm noded}}(\Delta,G)$, with noded family of arcs ${\mathcal F}$, and let
$w:\widehat{\mathbb C} \to \widehat{\mathbb C}$ be a 
noded quasiconformal deformation of $G$ with the complex dilatation $\mu$. Then,
the following statements are true
\begin{itemize}
\item[(1)] if $\alpha \in {\mathcal F}_{\mu}$, then 
both end points are the fixed points
of a loxodromic element of $G$. Such a loxodromic element keeps 
the connected component of $\Omega(G)$ containing $\alpha$
invariant;
\item[(2)] if $\Lambda(\mu) \neq \emptyset$, then 
$\Lambda(G) \subset \Lambda(\mu)$. This is a consequence of Proposition E.4 in \cite[page 96]{M1};
\item[(3)] $w(\Omega(u)) \cap w(\Lambda(\mu))=\emptyset$. Indeed, if there
were points $p_{1} \in \Omega(\mu)$ and $p_{2} \in \Lambda(\mu)$, such that
$w(p_{1})=w(p_{2})=q$, then by continuity of $w$ we could find two disjoint open
sets $U \subset \Omega(\mu)$ and $V$, $p_{1} \in U$, and $p_{2} \in V$, such that
$w(U)=w(V)$. The density property of $\Omega(\mu)$  asserts that there
are points $q_{1} \in U$ and $q_{2} \in V \cap \Omega(\mu)$ for which
$w(q_{1})=w(q_{2})$, contradicting the injectivity of the map $w$ restricted to
$\Omega(\mu)$;
\item[(4)] As a consequence of the definition of noded Beltrami differential, we have $w(\widehat{\mathbb C})=\widehat{\mathbb C}$, that is, the map
$w$ is surjective.
 \end{itemize}
 \end{rema}

\s

The importance of the noded quasiconformal deformations of Kleinian groups is
reflected in the following result.

\s
\noindent
\begin{theo}[\cite{Hidalgo-Vasiliev}]\label{Teo1} 
Let $\mu \in L^{\infty}_{{\rm noded}}(\Delta,G)$ and 
$w:\widehat{\mathbb C} \to \widehat{\mathbb C}$ be a 
noded quasiconformal deformation for $G$ with the complex dilatation $\mu$. Then
there is a unique Kleinian group $\theta(G)$ 
and  a unique isomorphism of groups $\theta:G \to \theta(G)$ such
that $w \circ \gamma = \theta(\gamma) \circ w$.  Moreover, the region of discontinuity of 
$\theta(G)$ is $w(\Omega(\mu) \cap \Omega(G))$.
\end{theo}

\s

The proof of the above theorem provided in \cite{Hidalgo-Vasiliev} also permit to have the following more general situation.

\s
\noindent
\begin{coro}[\cite{Hidalgo-Vasiliev}]\label{Cor1}
Let $\mu \in \overline{L^{\infty}_{1}(\Delta,G)}$ 
and let $w:\Omega(\mu) \to w(\Omega(\mu))$ be a
homeomorphism with complex dilatation $\mu$.
Suppose that there is a sequence $w_{n}$ of quasiconformal homeomorphisms
with corresponding complex dilatations $\mu_{n} \in L^{\infty}_{1}(\Delta,G)$,
converging locally uniformly to $w$ in $\Omega(\mu)$. Then
there exist a group $\theta(G)$ of M\"obius transformations and 
an isomorphism of groups $\theta:G \to \theta(G)$, 
such that $w \circ \gamma = \theta(\gamma) \circ w$.
\end{coro}

\s

The following result relates two noded quasiconformal deformations with the same dilation similar as the situation for the quasiconformal ones.

\s
\noindent
\begin{prop}\label{Lema2}
Let $\mu \in L^{\infty}_{noded}(\Delta,G)$ and let $w_{1}$ and $w_{2}$ be noded
quasiconformal deformations of $G$ with the complex dilatation $\mu$. Then there exists 
an orientation preserving homeomorphism  
$T:\widehat{\mathbb C} \to \widehat{\mathbb
C}$ whose restriction $T:w_{1}(\Omega(\mu)) \to w_{2}(\Omega(\mu))$ is a conformal mapping, such that
$T \circ w_{1} = w_{2}$.
\end{prop}
\begin{proof} The construction of $T$ is given as follows.
\begin{itemize}
\item[(3.1)] If $x \in w_{1}(\Omega(\mu))$, then set
$T(x)=w_{2}(w_{1}^{-1}(x))$.
\item[(3.2)] If $x \in w_{1}(\alpha_{n})$, then
set $T(x)=w_{2}(\alpha_{n})$.
\item[(3.3)] If $x \in
\Lambda(\mu)-\cup_{n}\alpha_{n}([0,1])$, then set $T(x)=w_{2}(w_{1}^{-1}(x))$.
\end{itemize}
\end{proof}

\s
\noindent
\begin{rema} 
Let $\mu \in L^{\infty}_{{\rm noded}}(\Delta,G)$, with associated noded family of arcs ${\mathcal F}$, 
and let
$w:\widehat{\mathbb C} \to \widehat{\mathbb C}$ be  a
noded quasiconformal deformation of $G$ with complex dilatation $\mu$. If
$\theta(G)$ and $\theta:G \to \theta(G)$ are as in Theorem \ref{Teo1}, then the following are easy to see.
\begin{itemize} 
\item[(1)] If $p$ belongs to one of the arcs in ${\mathcal F}$,
then  $w(p)$ is necessarily a doubly-cusped parabolic fixed point of $\theta(G)$.
\item[(2)] If $p \in \Lambda(G)$ is a loxodromic fixed point in $G$, 
which does not belong to any arc of ${\mathcal F}$, then $w(p)$ is again a 
loxodromic fixed point of $\theta(G)$.
\item[(3)] If $p$ is a rank two parabolic fixed point of $G$, then $w(p)$ is 
again a rank two parabolic fixed point in $\theta(G)$.
\item[(4)] If $p$ is a doubly-cusped parabolic fixed point of $G$, then $w(p)$ 
is again doubly-cusped in $\theta(G)$.
\end{itemize}
\end{rema}

\s

The previous remark permits to see the following fact.

\s
\noindent
\begin{theo}\label{Teo2} 
Under the hypothesis of Theorem \ref{Teo1}, if $G$ is
geometrically finite, then $\theta(G)$ is also geometrically finite.
\end{theo}

\s

The above, applied to Fuchsian groups of the first kind, permits to obtain the following with respect to noded Fuchsian groups.

\s
\noindent
\begin{coro}\label{coro4}
Let $G$ be a Fuchsian group of the first kind. Let $\mu \in L^{\infty}_{noded}({\mathbb D},G)$, let $w:\widehat{\mathbb C} \to \widehat{\mathbb C}$ be a  noded quasiconformal deformation for $G$ with the complex dilation $\mu$ and let $\theta:G \to \theta(G)$ such that $w \circ  \gamma=\theta(\gamma) \circ w$, for $\gamma \in G$. Then $\theta(G)$ is a noded Fuchsian group.
\end{coro}

\s

\subsection{Topological Realizations}\label{Sec:toporeal}
A simple loop $\alpha \subset S=\Delta/G$ (or a simple arc connecting two branch values of order $2$) is called {\it pinchable} if (a) its lifting on $\Omega(G)$
consists of pairwise disjoint simple arcs, and (b) each of these components is
stabilized by a primitive loxodromic transformation in $G$.

The stabilizer of a component of the lifting of a pinchable
loop $\alpha$ is a cyclic group generated by a primitive loxodromic
transformation, and the stabilizer of a component of the lifting of a
pinchable arc is a ${\mathbb Z}/2{\mathbb Z}$-extension of a cyclic group
generated by a primitive loxodromic transformation. We say that such a cyclic
group (or ${\mathbb Z}/2{\mathbb Z}$-extension) is defined by the pinchable
loop $\alpha$.

If $\beta_{1}$ and $\beta_{2}$ are two components of the lifting of a pinchable
$\alpha$, then their stabilizers are conjugate in $G$. 

We say that a
collection $\{\alpha_{1},...,\alpha_{m}\}$ of pairwise disjoint pinchable
loops or arcs is admissible if  they define non-conjugate 
groups in $G$ for $i \neq j$.

\s
\noindent
\begin{rema}\label{obsefuch}
In the particular case that $G$ is a Fuchsian group of the first kind and $\Delta={\mathbb D}$, then every homotopically non-trivial loop in $\Delta/G$ avoiding the branch locus is pinchable. Moreover, every collection ${\mathcal F}$ of pairwise disjoint pinchable loops  is admissible. Also, if $\Gamma$ is a finite index subgroup of $G$ and $Q:{\mathbb D}/\Gamma \to {\mathbb D}/G$ is a branched cover induced by the inclusion $\Gamma < G$, then the collection $Q^{-1}({\mathcal F})$ is admissible for $\Gamma$.

\end{rema}

\s

For each admissible collection ${\mathcal F}=\{\alpha_{1},...,\alpha_{m}\}$ of pinchable
loops on $S$ we define the following equivalence relation on the Riemann sphere
$\widehat{\mathbb C}$. Two points $p, q \in \widehat{\mathbb C}$
are equivalent if either:
\begin{itemize}
\item[(1)] $p=q$; or
\item[(2)] there is a component $\widetilde{\alpha_{j}}$ of the lifting of
some pinchable loop or arc $\alpha_{j}$, such that $p,q \in
\widetilde{\alpha_{j}} \cup \{a,b\}$, where  $a$ and $b$ are the endpoints of
$\widetilde{\alpha_{j}}$ (that is, the fixed points of a primitive
loxodromic transformation in the stabilizer of
$\widetilde{\alpha_{j}}$ in $G$)  
\end{itemize}

The set of equivalence classes for such an equivalence relation is
topologically the Riemann sphere. In fact, let us denote by $\widetilde{\mathcal F}$ the collection of all arcs (including their endpoints), as considered in (2) above, and let us consider the collection of continua given by the collection of arcs in $\widetilde{\mathcal F}$ as points and also each of the points in the complement of $\widetilde{\mathcal F}$. The discreteness of $G$ asserts that such collections of points is a semi-continuous collection of points, as defined in \cite{Moore}, and the result now follows from \cite[Thm2]{Moore}. 

Let us denote by
$P:\widehat{\mathbb C} \to \widehat{\mathbb C}$ the natural continuous
projection defined by the above relation. As a consequence of the results in \cite{M4} we have the following fact.

\s
\noindent
\begin{prop} \label{propo3}
There exists an orientation preserving homeomorphism
$Q:\widehat{\mathbb C} \to \widehat{\mathbb C}$, such that the map
$Q \circ  P:\widehat{\mathbb C} \to \widehat{\mathbb C}$ is a noded quasiconformal
deformation of the group $G$.
\end{prop}

\s

As a consequence of the above, the geometrically finite Kleinian groups constructed in \cite{KMS} are obtained by noded quasiconformal deformations of the suitable Kleinian groups. The following topological realization was seen in \cite{Hidalgo-Vasiliev}.

\s
\noindent
\begin{prop}[\cite{Hidalgo-Vasiliev}] \label{propo4}
Let $G$ be a Kleinian group and ${\mathcal F}$ be an admissible 
collection of pinchable loops. Denote by 
$P:\widehat{\mathbb C} \to \widehat{\mathbb C}$ the continuous projection naturally  induced by the equivalence relation defined by $\mathcal F$. Then, 
there is a discrete group $\theta(G)$ of orientation preserving homeomorphisms of
the Riemann sphere and there is an isomorphism $\theta:G \to \theta(G)$, such that $P \gamma =
\theta(\gamma) P$, for all $\gamma \in G$.
\end{prop}

\s
\noindent
\begin{rema} 
Results of \cite{M4} and \cite{Ohshika}
assert that, in Proposition \ref{propo4}, if $G$ is geometrically finite, then 
$\theta(G)$ is also geometrically finite. In fact, it can be shown that $\theta$ and $\theta(G)$ are as obtained in Theorem \ref{Teo1}, so the geometrically finiteness also follows from Theorem \ref{Teo2}.
\end{rema}

\subsection{The Noded Teichm\"uller space of $G$ supported in $\Delta$}\label{construccion}

\subsubsection{}
We may extend the Teichm\"uller equivalence relation, given previously on  $L^{\infty}_{1}(\Delta,G)$, to the whole $L^{\infty}_{{\rm noded}}(\Delta,G)$ as follows. Let $\mu,\nu \in L^{\infty}_{{\rm noded}}(\Delta,G)$ and $w_{\mu}$, $w_{\nu}$ be associated noded quasiconformal deformations for $G$, respectively. 
Theorem \ref{Teo1} asserts the existence of
isomorphisms $$\theta_{\mu}:G \to G_{\mu} \quad \mbox{and} \quad 
\theta_{\nu}:G \to G_{\nu}\;,$$ where 
$G_{\mu}$ and $G_{\nu}$ are Kleinian groups 
such that $w_{\mu} g = \theta_{\mu}(g) w_{\mu}$ and 
$w_{\nu} g = \theta_{\nu}(g) w_{\nu}$  for all $g \in G$. 
We say that $\mu$ and $\nu$ are {\it noded Teichm\"uller
equivalent} if there is an orientation preserving homeomorphism 
$A:\widehat{\mathbb C} \to \widehat{\mathbb C}$, such that 
\begin{itemize}
\item[(1)] $A(w_{\mu}(\Omega(\mu)))=w_{\nu}(\Omega(\nu))$;
\item[(2)] $A:w_{\mu}(\Omega(\mu)) \to w_{\nu}(\Omega(\nu))$ is conformal;
\item[(3)] $\theta_{\nu}(g)=A \theta_{\nu}(g) A^{-1}$ for all $g \in G$.
\end{itemize}

If in the above we replace (3) by $\theta_{\nu}(G)=A \theta_{\nu}(G) A^{-1}$, then we will say that $\mu$ and $\nu$ are {\it isomorphic}.

The {\it noded Teichm\"uller space of $G$ supported in $\Delta$}  is the set $NT(\Delta,G)$ of the noded Teichm\"uller equivalence classes of noded Beltrami coefficients
for $G$ supported in $\Delta$. If $\Delta=\Omega(G)$,  then we denote it by $NT(G)$.

\s
\noindent
\begin{rema} If $\mu, \nu \in L^{\infty}_{{\rm noded}}(\Delta,G)$ are noded
Teichm\"uller equivalent and $\mu \in L^{\infty}_{1}(\Delta,G)$, then (1) and
(2) in above definition assert that $\nu \in L^{\infty}_{1}(\Delta,G)$ and
that they are Teichm\"uller equivalent. Moreover,
the inclusion $L^{\infty}_{1}(\Delta,G) \subset L^{\infty}_{{\rm noded}}(\Delta,G)$ induces, under the above equivalence relation, 
the inclusion $T(\Delta,G) \subset NT(\Delta,G)$.
\end{rema}

\s

The modular group $M(\Delta,G)$ extends naturally to act on $NT(\Delta,G)$  
$$M(\Delta,G) \times NT(\Delta,G) \to NT(\Delta,G): ([w],[\mu]) \mapsto [\nu],$$
where $\nu \in L^{\infty}_{noded}(\Delta,G)$ is complex dilation of the noded quasiconformal deformation $w_{\mu} w^{-1}$. The quotient space
$N{\mathcal M}(\Delta,G)=NT(\Delta,G)/M(\Delta,G)$ is the space of isomorphic classes of noded Beltrami coefficients for $G$ supported at $\Delta$. It provides Deligne-Mumford's compactification of the moduli space ${\mathcal M}(\Delta,G)$.

\s
\noindent
\begin{rema} In the setting of representation of groups, $NT(G)$ corresponds to 
consider, in the representation space
Hom($G$,PGL($2,{\mathbb C}$)), the faithful representations of $G$ which are 
geometrically represented by noded quasiconformal deformations. 
This is a natural generalization for the deformation space of $G$,
on which one considers the geometric representations given by
quasiconformal homeomorphisms of $G$.
\end{rema}

\s
\subsection{The case of Fuchsian groups: Augmented Teichm\"uller space}\label{Ex5}
Let us assume in here that $G$ is a Fuchsian group of the first kind without parabolic elements, keeping the unit disc $\Delta_{1}={\mathbb D}$ invariant, and set 
$\Delta_{2}=\widehat{\mathbb C}-\overline{\Delta_{1}}$. In this case, $\Omega(G)=\Delta_{1} \cup \Delta_{2}$ and 
the Riemann orbifolds $\Delta_{1}/G$ and $\Delta_{2}/G$ are of signature $(g;n_{1},\stackrel{r}{\ldots},n_{r})$, where $2 \leq n_{1} \leq \cdots \leq n_{r}<\infty$ and 
$n_{1}^{-1}+\cdots+n_{r}^{-1}<2g+r-2$.

Associated to $G$ we have the Teichm\"ullers spaces $T(\Delta_{1},G)$, $T(\Delta_{2},G)$ and $T(G)=T(\Omega(G),G)$. These are 
simply connected complex manifolds, the first two isomorphic and of complex dimensions $3g-3+r$, and the last one of complex dimension $6g-6+2r$. In fact, there is an isomorphism
$$\theta:T(G) \to T(\Delta_{1},G) \times T(\Delta_{2},G): [\mu] \mapsto ([\mu_{1}],[\mu_{2}]),$$
where $\mu_{i}$ is defined as $\mu$ on $\Delta_{i}$ and zero on its complement.

Associated to the above three spaces are the corresponding partial closures given by the noded Teichm\"uller spaces $NT(\Delta_{1},G)$, $NT(\Delta_{2},G)$ and $NT(G)$. Each $NT(\Delta_{j},G)$ can be identified with the Augmented Teichm\"uller
space of $G$ as defined in \cite{Ab1,Ab2,Bers3}. The isomorphism $\theta$ extends continuously to an isomorphism
$$\theta:NT(G) \to NT(\Delta_{1},G) \times NT(\Delta_{2},G).$$

The Teichm\"uller space $T(\Delta_{j},G)$ has associated the Weil-Petersson (WP) metric, which is K\"ahler, has negative sectional curvature, is not complete and for every pair of points there is a unique geodesic connecting them\cite{Ahlfors,Royden,Tromba,Wolpert1,Wolpert2}. It is known that the WP metric completion of $T(\Delta_{j},G)$ is the augmented Teichm\"uller space $NT(\Delta_{j},G)$ \cite{Masur} (but it is a non-locally compact space). 
Results on the geometry of geodesics on these spaces is given in \cite{Wolpert3}. 

The modular group $M(\Delta_{j},G)$ acts a group of WP orientation-preserving isometries (these are all these isometries if $(g,r) \neq (1,2)$ \cite{Wolpert3}). 
It was observed by Masur that $NT(\Delta_{j},G)/M(\Delta_{j},G)$ is the quotient WP metric completion $\overline{\mathcal M}(\Delta_{j},G)$ (this being the Deligne-Mumford's compactification) of the moduli space ${\mathcal M}(\Delta_{j},G)$ (the space of isomorphic classes of Riemann surfaces of genus $g$ with $r$ punctures).

In \cite{HV} it was proved that if $M$ is a finite index subgroup of the modular group $M({\mathbb D},G)$, then the quotient spaces $NT(\Delta_{1},G)/M$ and $NT(\Delta_{2},G)/M$ are compact complex orbifolds. In particular, $NT(\Omega(G)/M$ is also a compact complex orbifold.

\s
\noindent
\begin{rema}
Now, following Remark \ref{obs2} and the above, for every pair $(\Delta,G)$, where $G$ is a finitely generated Kleinian group, and $M$ being of finite index in $M(\Delta,G)$, it holds that $NT(\Delta,G)/M$ is a compact complex orbifold.
\end{rema}

\s

\section{Uniformizations of $(\gamma,n)$-gonal pairs: proof of Theorem \ref{main1}} \label{Sec:uniformizationsofpairs}
Let us consider a $(\gamma,n)$-gonal pair $(S,f)$ and 
let us denote by $B_{f}\subset R$ the locus of branch values of $f$. 

\subsection{The signature of $(S,f)$}
For each point in $B_{f}$ we define its {\it branch order} as the minimum common multiple of the local degrees at all its preimages under $f$. 
Let us write $B_{f}=\{p_{1},\ldots,p_{r}\}$ so that $n_{j}$ is the branch order of $p_{j}$, where $2 \leq n_{1} \leq n_{2} \leq \cdots \leq n_{r}$.
The {\it signature} of $(S,f)$ is then defined by the tuple $(\gamma;n_{1},\ldots,n_{r})$. We assume, from now on, that $(S,f)$ is of {\it hyperbolic type}, that is,  
$n_{1}^{-1}+\cdots+n_{r}^{-1}<2\gamma+r-2$ (in particular, if $\gamma=0$, then $r \geq 3$). If $S$ has genus $g$, then  Riemann-Hurwitz formula asserts the equality
$$2(g-1-n(\gamma-1))=\sum_{j=1}^{r}\left(n-\# f^{-1}(p_{j})\right).$$

Topologically equivalent $(\gamma,n)$-gonal pairs have the same signature, but the converse is in general false.

\subsection{A pair of Riemann orbifolds}
For each point $q \in f^{-1}(p_{j})$ set
$m_{q}=n_{j}/d_{q}$, where $d_{q}$ is the local degree of $f$ at $q$, and let $N_{f}$ be the set of of points $q \in f^{-1}(B_{f})$ with $m_{q}>1$.
Let $R^{orb}$ be the Riemann orbifold whose underlying Riemann surface structure is $R$, its cone point set is $B_{f}$ and the cone order of each $p_{j}$ is $n_{j}$. Similarly, let $S^{orb}$ be the Riemann orbifold whose underlying Riemann surface structure is $S$ and its cone points are given by the points $q \in N_{f}$ with cone order $m_{q}>1$.

\subsection{An uniformizing pair}
As a consequence of the uniformization theorem: (i) there is a co-compact Fuchsian group $G$ acting on the unit disc ${\mathbb D}$ and an orbifold isomorphism $\psi:R^{orb} \to {\mathbb D}/G$; that is, $G$ has signature $(\gamma;n_{1},\ldots,n_{r})$, and (ii) there is an index $n$ subgroup $\Gamma$ and an orbifold isomorphism $\phi:S^{orb} \to {\mathbb D}/\Gamma$  so that $\psi f \phi^{-1}:{\mathbb D}/\Gamma \to {\mathbb D}/G$ is being induced by the inclusion $\Gamma < G$.

The constructed pair of Fuchsian groups $(\Gamma,G)$ is uniquely determined by the pair $(S,f)$, up to conjugation by elements in ${\rm Aut}({\mathbb D})$, the group of holomorphic automorphisms of ${\mathbb D}$; we say that $(\Gamma,G)$ {\it uniformizes} $(S,f)$. The following fact follows almost immediately from the definitions.

\s
\noindent
\begin{lemm}\label{lema2}
Two pairs of Fuchsian groups $(\Gamma_{1},G_{1})$ and $(\Gamma_{2},G_{2})$ uniformize topologically equivalent $(\gamma,n)$-gonal pairs (of hyperbolic type) if and only if there is an orientation-preserving homeomorphism $u:{\mathbb D} \to {\mathbb D}$ so that $u G_{1} u^{-1}=G_{2}$ and $u \Gamma_{1} u^{-1}=\Gamma_{2}$ (as the orbifolds are compact, we may assume $u$ to be a quasiconformal homeomorphism). They correspond to twisted isomorphic pairs if $u \in {\rm Aut}({\mathbb D})$ and to isomorphic ones if $G_{1}=G_{2}$ and $u \in G_{2}$. 
\end{lemm}

\s
\noindent
\begin{rema}
Let us recall that a holomorphic automorphism of $(S,f)$ is a holomorphic automorphism $T$ of $S$ so that $f=f T$. In this context, 
the group of holomorphic automorphisms of $(S,f)$ is naturally identified with the quotient $N_{G}(\Gamma)/\Gamma$, where $N_{G}(\Gamma)$ is the normalizer of $\Gamma$ in $G$.
\end{rema}

\s

As a consequence of Lemma \ref{lema2} one obtains that  ${\mathcal H}_{0}(S,f)$ can be identified with 
the Teichm\"uller space $T({\mathbb D},G)$ of the orbifold ${\mathbb D}/G$, which is a simply-connected complex manifold of dimension $3\gamma+ r-3$ \cite{K-M,L,N}. 

\s

As every point in ${\mathcal H}(S,f)$ is uniformized by a quasiconformal deformation of $(\Gamma,G)$, is a consequence of the measurable Riemann mapping theorem \cite{A-B,Morrey} that the Hurwitz space ${\mathcal H}(S,f)$ is connected and it can be identified with 
$T({\mathbb D},G)/M(\Gamma,G)$, where $M(\Gamma,G)$ is a subgroup of the modular group $M({\mathbb D},G)$ (the group of holomorphic automorphisms of $T({\mathbb D},G)$) induced by the isotopic classes of those quasiconformal self-homeomorphisms of ${\mathbb D}$ normalizing both $G$ and also $\Gamma$. The above, in particular, asserts that ${\mathcal H}(S,f)$ is a connected complex orbifold of dimension $3\gamma+r-3$ whose orbifold fundamental group is isomorphic to $M(\Gamma,G)$. 

\s
\noindent
\begin{rema}
As the subgroup $M(\Gamma,G)$ has finite index in the modular group $M({\mathbb D},G)$, there is a finite degree branched covering ${\mathcal H}(S,f) \to {\mathcal M}({\mathbb D},G)$, where
${\mathcal M}({\mathbb D},G)$ is the moduli space of the orbifold ${\mathbb D}/G$ (this being the moduli space of an $r$-punctured Riemann surface of genus $\gamma$). For instance, if $\Gamma$ is a characteristic subgroup of $G$, then $M(\Gamma,G)=M({\mathbb D},G)$, so  ${\mathcal H}(S,f)={\mathcal M}({\mathbb D},G)$.
\end{rema}

\s

\section{Uniformizations of stable $(\gamma,n)$-gonal pairs: Proof of Theorem \ref{main2}}\label{Sec:nodedpairs}
In this section we proceed to construct uniformizations of an stable $(\gamma,n)$-gonal pair  using noded Fuchsian groups. These uniformization permits to provide a compactification $\overline{\mathcal H}(S,f)$, for $(S,f)$ a $(\gamma,n)$-gonal pair of hyperbolic type, this being a complex orbifold containing ${\mathcal H}(S,f)$ as an open dense suborbifold.

\subsection{}
Let us recall that each stable $(\gamma,n)$-gonal pair $(S^{*},f^{*})$ is obtained by considering a suitable $(\gamma,n)$-gonal pair $(S,f:S \to R)$ of hyperbolic type and a suitable collection ${\mathcal F}=\{\gamma_{1}, \ldots, \gamma_{s} \}$ of pairwise disjoint simple loops in $R-B_{f}$, where $B_{f}$ is the set of branch values of $f$. If $r$ is the cardinality of $B_{f}$, then for $\gamma=0$ we also assume that $r \geq 4$ (otherwise, there is no possible degeneration to make). 

The collection ${\mathcal F}$ satisfies that the Euler characteristic of each connected component of $R-(B_{f} \cup \gamma_{1}\cup \cdots \cup \gamma_{s})$ is negative and none of the components of $R-(\gamma_{1} \cup \cdots \cup \gamma_{s})$ is a disc with only two branch values, both with branch order equal to $2$.  

By the pinching process of the loops in ${\mathcal F}$ and those in $f^{-1}({\mathcal F})$ (as indicated in the introduction), we obtain a topological stable $(\gamma,n)$-gonal pair homeomorphic to $(S^{*},f^{*})$. To obtain $(S^{*},f^{*})$ we need to provide to each component  of the nodes of a suitable analytically finite Riemann surface.

\subsection{}
Let us consider a pair $(\Gamma,G)$ of Fuchsian groups uniformizing $(S,f)$, that is, there are orbifold isomorphisms 
$$\phi:S^{orb} \to {\mathbb D}/\Gamma, \quad \psi:R^{orb} \to {\mathbb D}/G$$
so that $ \pi:=\psi  f \phi^{-1}$ is branched cover induced by the inclusion of $\Gamma$ in $G$. Recall that the pair $(\Gamma,G)$ is unique up to conjugation by elements of ${\rm Aut}({\mathbb D})$.

\subsection{}
The collection of simple loops $\psi({\mathcal F}) \subset {\mathbb D}/G$ lifts under $\pi$ to a collection ${\mathcal G}$ of simple arcs in ${\mathbb D}$ which is pinchable for $G$, so also for the subgroup $\Gamma$ (see Remark \ref{obsefuch}).

\subsection{}
Set $\overline{\mathbb D}=\{z \in {\mathbb C}: |z|>1\} \cup \{\infty\}$ and let us consider three different points $a,b,c \in \overline{\mathbb D}$.

\subsection{}
Now, as a consequence of Propositions \ref{propo3} and \ref{propo4}, there is a $ \mu \in L^{\infty}_{noded}({\mathbb D},G)$ with limit set being the closure of the collection ${\mathcal G}$. If $w_{\mu}:\widehat{\mathbb C} \to \widehat{\mathbb C}$ is a noded quasiconformal deformation for $G$ with the complex dilation $\mu$ and fixing the points $a,b,c$, then Theorem \ref{Teo1} asserts the existence of a unique Kleinian group $\theta(G)$ and  a unique isomorphism of groups $\theta:G \to \theta(G)$ such
that $w_{\mu} \circ \gamma = \theta(\gamma) \circ w_{\mu}$.  Moreover, the region of discontinuity for the action of 
$\theta(G)$ on the Riemann sphere is $w_{\mu}(\Omega(\mu) \cap \Omega(G))=w_{\mu}(\Omega(\mu) \cap {\mathbb D}) \cup w_{\mu}(\overline{\mathbb D})$. Observe that $w_{\mu}(\overline{\mathbb D})$ is a quasidisc.

Corollary \ref{coro4} asserts that $\theta(G)$ (and so $\theta(\Gamma)$) is a noded Fuchsian group. By Proposition \ref{propo3}, the pair $(\theta(\Gamma),\theta(G))$, when restricted to $w_{\mu}(\Omega(\mu) \cap {\mathbb D})$, provides a stable $(\gamma,n)$-gonal pair modelled by $(S^{*},f^{*})$. Also, the same pair, when restricted to $w_{\mu}(\overline{\mathbb D})$ provides the $(\gamma,n)$-gonal pair $(S,f)$.

Up to post-composition by a suitable quasiconformal homeomorphisms, we may assume that the stable $(\gamma,n)$-gonal pair we obtain is the original one, i.e., $(S^{*},f^{*})$; 
we say that the pair $(\theta(\Gamma),\theta(G))$ {\it uniformizes} it.

\subsection{}
For $t \in {\mathbb D}$, we may consider $t\mu \in  L^{\infty}_{1}({\mathbb D},G)$ and the quasiconformal homeomorphism 
$w_{t\mu}:\widehat{\mathbb C} \to \widehat{\mathbb C}$ with complex dilation $t\mu$ and fixing the points $a,b,c$. If $G_{t}=w_{t\mu} G w_{t\mu}^{-1}$ and 
$\Gamma_{t}=w_{t\mu} \Gamma w_{t\mu}^{-1}$, then $(\Gamma_{t},G_{t})$ provides a continuous family of $(\gamma,n)$-gonal pairs $(S_{t},f_{t})$ converging to $(S^{*},f^{*})$.

\subsection{}
The above construction asserts that the Hurwitz space $\overline{\mathcal H}_{0}(S,f)$, parametrizing isomorphic classes of stable $(\gamma,n)$-gonal pairs, modelled by degeneration of a the $(\gamma,n)$-gonal pair $(S,f)$, can be identified with the noded Teichm\"uller space $NT({\mathbb D},G)$. 

Similarly, the Hurwitz space $\overline{\mathcal H}(S,f)$, parametrizing twisted isomorphic classes of stable $(\gamma,n)$-gonal pairs, modelled by degeneration of a the $(\gamma,n)$-gonal pair $(S,f)$, can be identified with the quotient space $NT({\mathbb D},G)/M(\Gamma,G)$. As $M(\Gamma,G)$ has finite index in the modular group $M({\mathbb D},G)$, it follows from the results of  \cite{HV} that $NT({\mathbb D},G)/M(\Gamma,G)$ carries a structure of a complex orbifold, being a finite branched cover of the Deligne-Mumford's compactification space of ${\mathbb D}/G$.

\s
\noindent
\begin{rema}
The forgetful map (of finite degree) 
$${\mathcal H}(S,f) \to {\mathcal M}_{g}: [(S',f')] \mapsto [S']$$ extends to the corresponding forgetful map (also of finite degree) 
$$\overline{\mathcal H}(S,f) \to \widehat{\mathcal M}_{g}: [(S',f')] \mapsto [S']$$
\end{rema}

\s



\begin{thebibliography}{99}
\bibitem{Ab1}
W. Abikoff. On boundaries of Teichm\"uller spaces and on Kleinian groups, III.
{\it Acta Math.} {\bf 134} (1975), 211--237.

\bibitem{Ab2}
W. Abikoff. 
Degenerating families of Riemann surfaces. {\it Ann. of Math.} (2)
{\bf 105} (1977), 29-44.


\bibitem{Abikoff}
W. Abikoff.
Augmented Teichm\"uller spaces.
{\it Bull. of the Amer. Math. Soc.} {\bf 82} No. 2 (1976), 333--334.


\bibitem{Ahlfors:finitud}
L. V.  Ahlfors.
Finitely generated Kleinian groups.
{\it Amer. J. Math.} {\bf 86} (1964), 413-429; correction {\it ibid} {\bf 87} (1965), 759. 

\bibitem{Ahlfors}
L. V.  Ahlfors. 
Some remarks on Teichm\"uller's space of Riemann surfaces.
{\it Ann. of Math.} (2) {\bf 74} (1961), 171--191. 



\bibitem{A-B}
Ahlfors, L. and Bers, L.
Riemann's mapping theorem for variable metrics.
{\it Ann. Math. } {\bf 72} (1960), 385--404.


\bibitem{Arfeux}
M. Arfeux.
Berkovich spaces and Deligne-Mumford compactification.
https://arxiv.org/abs/1506.02552

\bibitem{BCI}
G. Bartolini, A. F. Costa and M. Izquierdo. 
On automorphisms groups of cyclic $p$-gonal Riemann surfaces.
{\it Journal of symbolic computation} {\bf 57} (2013), 61--69. 

\bibitem{Beardon}
A.F. Beardon. {\it The geometry of discrete groups}. Graduate Texts in
Mathematics, {\bf 91}. Springer-Verlag, New York.


\bibitem{Belyi}
G. V. Belyi. \textit{On Galois Extensions of a Maximal cyclotomic field}. Math. USSR Izvestija  {\bf 14}, (1980), 247--265.



\bibitem{Bers}
L. Bers.
Simultaneous uniformization.
{\it Bull. of the Amer. Math. Soc.} {\bf 66} No. 2 (1960),  94--97.


\bibitem{Bers2}
L. Bers.
Spaces of Kleinian groups. 
Several Complex Variables, Maryland 1970.
{\it Lecture Notes in Mathematics} {\bf 155} (1970), Springer, Berlin, pp. 9--34


\bibitem{Bers3}
L. Bers.
On boundaries of Teichm\"uller spaces and on Kleinian groups. I.
{\it  Annals of Mathematics. Second Series} {\bf 91}(1970),  570--600.


\bibitem{BC}
E. Bujalance and A. F. Costa.
Automorphism Groups of Cyclic $p$-gonal Pseudo-real Riemann Surfaces.
{\it Journal of Algebra} {\bf 440} (2015), 531--344.

\bibitem{Clebsch}
A. Clebsch. 
Zur Theorie der Riemann'schen Fl\"ache. 
{\it Mathematische Annalen} {\bf 6} No. 2 (1873), 216--230.

\bibitem{Co-Se}
J. Coelho and F. Sercio.
On the gonality of stable curves.
https://arxiv.org/pdf/1507.07494v2.pdf




\bibitem{Fried}
M. Fried.
Fields of definition of function fields and Hurwitz families -- groups as Galois groups.
{\it Comm. Alg.} {\bf 5} No. 1 (1977), 17--82.

\bibitem{GWW}
G. Gromadzki, A. Weaver and A. Wootton.
On gonality of Riemann surfaces.
{\it Geom Dedicata} {\bf 149} (2010),1--14.


\bibitem{H-M}
J. Harris and D. Mumford.
On the Kodaira dimension of the moduli space of curves. 
{\it Invent. Math.} {\bf 67} (1982), 23--86.

\bibitem{Hidalgo:nodedSchottky}
R. A. Hidalgo.
The noded Schottky Space.
{\it Proc. London Math Soc.} {\bf 73} (1996), 385--403.



\bibitem{Hidalgo:nodedFuchsian}
R. A. Hidalgo.
Noded Fuchsian groups I.
{\it Complex Variables: Theory and Appl.} {\bf 36} (1998), 45--66.




\bibitem{Hidalgo-Vasiliev}
R. A. Hidalgo and A. Vasil'ev.
Noded Teichm\"uller spaces.
{\it Journal d'Analyse Mathematique} {\bf 99} No. 1 (2006), 89--107.

\bibitem{HV}
V. Hinin and A. Vaintrob.
Augmented Teichm\"uller spaces and orbifolds.
https://arxiv.org/pdf/0705.2859.pdf

\bibitem{Hurwitz}
A. Hurwitz.
\"Uber Riemann'sche Fl\"achen, mit gegebenen Verzweigungspunkten. 
{\it Math. Ann.} {\bf 39}, No 1 (1891), 1--60.


\bibitem{KMS}
L. Keen, B. Maskit and C. Series. 
Geometric finiteness and uniqueness
for Kleinian groups with circle packing limit sets. 
{\it J. Reine Angew. Math.} {\bf 436} (1993), 209-219.

\bibitem{Koebe} 
P. Koebe. 
\"{U}ber die Uniformisierung der Algebraischen Kurven II. 
{\it Math. Ann.} {\bf 69} (1910),1--81.

\bibitem{Kra}
I. Kra. 
On spaces of Kleinian groups. 
{\it Comment. Math. Helv.} {\bf 47} (1972), 53--69.

\bibitem{K-M}
I. Kra and B. Maskit. 
The deformation space of a Kleinian group.
{\it Amer. J. Math.} {\bf 103} (1981), 1065-1102.


\bibitem{Kra-Maskit}
I. Kra and B. Maskit. 
Pinching two component Kleinian groups. {\it Analysis
and Topology}, World Scientific Press, 1998, 425-465.


\bibitem{L}
O. Letho. 
{\it Univalent Functions and Teichm\"uller Spaces.} Graduate Texts in
Mathematics {\bf 109}, Springer-Verlag 1986.


\bibitem{M1}
B. Maskit. 
{\it Kleinian groups.} Grundlehren der Mathematischen
Wissenschaften, vol. 287,Springer-Verlag, Berlin, Heildelberg,
New York, 1988.


\bibitem{M2}
B. Maskit. 
On boundaries of Teichm\"uller spaces II. {\it Ann. of Math.} {\bf 91}
(1970), 607-639.

\bibitem{M3}
B. Maskit.
On a Class of Kleinian Groups. 
{\it Ann. Acad. Sci. Fenn. Ser. A I Math.} {\bf 442} (1969).\\
http://www.acadsci.fi/mathematica/1969/no442pp01-08.pdf


\bibitem{M4}
B. Maskit. 
Parabolic elements in Kleinian groups. {\it Annals of Math.} {\bf
117} (1983), 659-668.

\bibitem{M5}
B. Maskit.
Self-maps of Kleinian groups.
{\it Amer. J. Math.} {\bf 93} (1971), 840--856.

\bibitem{Masur}
H. Masur.
Extension of the Weil-Petersson metric to the boundary of Teichm\"uller space.
{\it Duke Math. J.} {\bf 43} No. 3 (1976), 623--635.

\bibitem{Moore}
R. L. Moore. 
Concerning upper semi-continuous collection of continua. {\it
Trans. AMS} {\bf 27} (1925), 412-428.


\bibitem{Morrey}
Morrey, Charles B. Jr. 
On the Solutions of Quasi-Linear Elliptic Partial Differential Equations.
{\it Trans. of the Amer. Math. Soc.} {\bf 43} (1) (1938), 126--166



\bibitem{N}
S. Nag. 
{\it The Complex Analytic Theory of Teichm\"uller Spaces.} 
Canadian Mathematical Society,  Series of Monographs and Advances Texts. 
A Wiley Interscience Pub. 1988.

\bibitem{Natanzon1}
S. M. Natanzon. 
Uniformization of spaces of meromorphic functions.
{\it Soviet Math. Dokl.} {\bf 33} (1986) 487--490.

\bibitem{Natanzon2}
S. M. Natanzon.
Fuchsian groups and unifromization of Hurwitz spaces.
Disertaciones del Seminario de Matem\'aticas Fundamentales {\bf 19}, 2013, 
Universidad Nacional de Educaci\'on a Distancia.
https://www.researchgate.net/publication/28272129

\bibitem{Ohshika}
K. Ohshika. 
Geometrically finite Kleinian groups and parabolic elements.
{\it Proc. Edinburgh Math. Soc.} (2) {\bf 41} (1998), no. 1, 141--159.

\bibitem{Royden}
H. L. Royden.
Intrinsic metrics on Teichm\"uller space.
In {\it Proceedings of the International Congress of Mathematicians} (Vancouver, B. C., 1974), {\bf 2}, pp. 217--221.
Canad. Math. Congress, Montreal, Que., 1975.

\bibitem{Severi}
F. Severi.
{\it Vorlesungen \"uber algebraische Geometrie}. 
Teubner-Verlag, 1921.

\bibitem{Sullivan}
D. Sullivan. 
On the ergodic theory at infinity of an arbitrary discrete group of hyperbolic motions, 
Riemann Surfaces and Related Topics. 
{\it Ann. of Math. Studies} {\bf 97} (1981), 465--496


\bibitem{Thurston1988}
W. P. Thurston. 
On the geometry and dynamics of diffeomorphisms of surfaces.
{\it American Mathematical Society. Bulletin. New Series} {\bf 19} No. 2 (1988), 417--431.

\bibitem{Tromba}
A. J. Tromba.
On a natural algebraic affine connection on the space of almost complex structures and the curvature of Teichm\"uller space with respect to its Weil-Petersson metric.
{\it Manuscripta Math.} {\bf 56} No. 4 (1986), 475--497.


\bibitem{Wolpert1}
S. A. Wolpert.
Chern forms and the Riemann tensor for the moduli space of curves.
{\it Invent. Math.} {\bf 85} No. 1 (1986), 119--145.

\bibitem{Wolpert2}
S. A. Wolpert.
Geodesic length functions and the Nielsen problem.
{\it J. Differential Geom.} {\bf 25} No. 2 (1987), 275--296.

\bibitem{Wolpert3}
S. A. Wolpert.
Geometry of the Weil-Petersson completion of the Teichm\"uller space.
{\it Surveys in Differential Geometry} {\bf 8} (2003), 357--393.

\end{thebibliography}
\end{document}